\documentclass[12pt]{article}
\usepackage{fancyhdr,graphicx}
\usepackage{amsfonts}
\usepackage{amssymb}
\usepackage{amsthm}
\usepackage{newlfont}
\usepackage{amsmath}

\usepackage{bbding}   
\usepackage{color,xcolor}

\newtheorem{thm}{Theorem}[section]
\newtheorem{Lemma}[thm]{Lemma}
\newtheorem{cor}[thm]{Corollary}
\newtheorem{pro}[thm]{Proposition}

\usepackage{latexsym, bm}
\title{Decomposition theorem on matchable distributive lattices\thanks{This
 work was supported by NSFC (grant no. 10831001).}}
\author{Heping Zhang$^1$, Dewu Yang$^{1,2}$ and Haiyuan Yao$^1$}
\date{\small $^1$School of Mathematics and Statistics, Lanzhou
University, Lanzhou, Gansu 730000,\\ P. R. China,
zhanghp@lzu.edu.cn, hyyao@lzu.edu.cn.\\\vspace{0.2cm} \small
$^2$School of Mathematics and Statistics,  Henan University of
Science and Technology, Luoyang,\\ Henan 471003, P. R. China,
dewuyang0930@163.com}

\begin{document}
\setlength{\baselineskip}{18pt} \maketitle

\begin{abstract}
A distributive lattice structure ${\mathbf M}(G)$ has been
established on the set of perfect matchings of a plane bipartite
graph $G$. We call a lattice  {\em matchable distributive lattice}
(simply MDL) if it is isomorphic to such a distributive lattice. It
is natural to ask which lattices are  MDLs. We show that if a plane
bipartite graph $G$ is elementary, then ${\mathbf M}(G)$ is
irreducible. Based on this result, a decomposition theorem on MDLs
is obtained: a finite distributive lattice $\mathbf{L}$ is an MDL if
and only if each factor in any cartesian product decomposition of
$\mathbf{L}$ is an MDL. Two types of MDLs are presented:
$J(\mathbf{m}\times \mathbf{n})$ and $J(\mathbf{T})$, where
$\mathbf{m}\times \mathbf{n}$ denotes the cartesian product between
$m$-element chain and $n$-element chain, and $\mathbf{T}$ is a poset
implied by any orientation of a tree. \vspace{0.4cm}


\end{abstract}


\section{Introduction}

Perfect matching of graphs is significant for theoretical chemistry
and theoretical physics. This graph-theoretical concept coincides
with that of the Kekul\'{e} structure of organic molecules. The
Kekul\'e structure count can be used to predict the stability of
benzenoid hydrocarbons. The carbon-skeleton of a benzenoid
hydrocarbon is a hexagonal system, i.e. 2-connected plane graph
every interior face of which is a  regular hexagon of side length
unit. Since $1980'$s there have been developed a combinatorial
object,  the $Z$-transformation graph (or resonance graph)
\cite{zgc, zgc1} on the set perfect matchings of a hexagonal system,
 late extended to a general plane bipartite graph \cite{zh,zz4,zz2,
zzy}; see a recent survey \cite{Zh2}. Randi\'c \cite{r,r2} showed
that the leading eigenvalue
 of the resonance graphs has a quite satisfactory
correlation  with the resonance energy of benzenoid hydrocarbons.

A domino tiling of a polygon in the plane corresponds to a perfect
matching of a related graph. In theoretical physics, a domino is
seen as a dimer, a diatomic molecule (as the molecule of hydrogen),
and each tiling is seen as a possible state of a solid or a fluid.
In 2003  Fournier \cite{fj} reintroduced Z-transformation graph
under name ``perfect matching graph" in investigating domino tiling
spaces of Saldanha et al. \cite{st}. E. R\'emila \cite{re,re2}
established
 the distributive lattice structure on the set of domino
tilings of a polygon by using Thurston's {\em height function}. In
general, a distributive lattice on the set of perfect matchings of a
plane bipartite graph was presented in terms of Z-transformation
digraph and the unit decomposition of alternating cycle systems with
respect to a perfect matching \cite{lz}.


Let $G$ be a finite and simple graph with vertex-set $V(G)$ and
edge-set $E(G)$. A {\em perfect matching} or {\em 1-factor} of $G$
is a set of independent edges which saturate all vertices of $G$.
Let $\mathcal{M}(G)$ denote the set of 1-factors of $G$. A plane
bipartite graph $G$ is {\em elementary} \cite{lm} if $G$ is
connected and every edge is contained in some 1-factor; further {\em
weakly elementary} \cite{zz2,wc} if every alternating cycle with
respect to some 1-factor together with its interior form an
elementary subgraph.

For a plane bipartite graph $G$, the $Z$-transformation graph $Z(G)$
is defined as a graph on $\mathcal{M}(G)$: $M, M'\in \mathcal{M}(G)$
are joined by an edge if and only if they differ only in one cycle
that is the boundary of an inner face of $G$.

To give an acyclic orientation of $Z(G)$ \cite{zz4}, a proper
2-coloring (white-black) of  bipartite graph $G$ is specified. For
$M\in \mathcal{M}(G)$, a cycle $C$ is said to be $M$-{\em
alternating} if the edges of $C$ appear alternately in and off $M$;
further {\em proper} ({\em improper}) \cite{zz3} if every edge of
$C$ belonging to $M$ goes from white (black) end-vertex to black
(white) end-vertex along the clockwise orientation of $C$. Now
Z-transformation digraph $\vec{Z}(G)$ is the orientation of $Z(G)$:
 an edge $M_1M_2$ of $Z(G)$ is oriented from  $M_1$ to $M_2$
if the symmetric difference $M_1\oplus M_2$ form a proper $M_1$- and
improper $M_2$-alternating cycle (the boundary of an inner face).

Since $\vec{Z}(G)$ has no directed cycles \cite{zz4}, it naturally
implies a partial ordering on $\mathcal{M}(G)$. This poset is
denoted by $\mathbf{M}(G)$. Then its Hasse diagram is isomorphic to
$\vec{Z}(G)$. Lam and Zhang \cite{lz} showed that $\mathbf{M}(G)$ is
a finite distributive lattice (FDL)  if $G$ is weakly elementary.
Further the first author of the present paper showed \cite{zh1} that
$\mathbf{M}(G)$ is direct sum of at least two distributive lattices
if $G$ is non-weakly elementary. By applying such a lattice
structure, Zhang et al. showed \cite{ZLS} that every connected
resonance graph of plane bipartite graphs is a median graph, and
extended Klav\v{z}ar et al.'s result \cite{rcm}  in the case of
cata-condensed benzenoid systems.

In different ways, Propp \cite{pj} established  a distributive
lattice structure on the set of $c$-orientations of a plane
bipartite graph $G$; Pretzel \cite{po} provided a new proof to
Propp's result. Similar structures were also given on the set of
reachable configurations of an edge firing game \cite{mc},
$\alpha$-orientations of a planar graph \cite{fs}, and flows of a
planar graph \cite{knk}.

In this paper we propose a  problem: which  distributive lattices
are isomorphic to distributive lattice $\mathbf{M}(G)$ on the set of
1-factors of a plane bipartite graphs $G$?  A lattice is called {\em
matchable distributive lattice} (simply MDL) if it is isomorphic to
such a distributive lattice $\mathbf{M}(G)$. Non-matchable
distributive lattices exist.
 We show that if a plane bipartite graph $G$ is elementary, then
${\mathbf M}(G)$ is irreducible. Based on this result, a
decomposition theorem on MDL is obtained (Theorem \ref{iff}): a
finite distributive lattice $\mathbf{L}$ is an MDL if and only if
each factor in any cartesian product decomposition of $\mathbf{L}$
is an MDL. Finally, we present two types of irreducible MDLs by
applying the {\em fundamental theorem for finite distributive
lattices} (FTFDL): $J(\mathbf{m}\times \mathbf{n})$ and
$J(\mathbf{T})$, where $\mathbf{m}\times \mathbf{n}$ denotes the
cartesian product between
 $m$-element chain and $n$-element chain, and $\mathbf{T}$ is a
poset implied by any orientation of a tree. Meantime,  we also show
that for  any order ideal $W$ of $\mathbf{m}\times \mathbf{n}$,
$J(W)$ is an MDL.

\section{Preliminaries}

Terms on poset and distributive lattice used in this paper can be
found in \cite{gb, gg, ec}. If $\mathbf{P}$ and $\mathbf{Q}$ are
posets, then the direct (cartesian) product of $\mathbf{P}$ and
$\mathbf{Q}$ is the poset $\mathbf{P}\times \mathbf{Q}$ on the set
$\{(x,y): x\in \mathbf{P}$ and $y\in \mathbf{Q}\}$ such that
$(x,y)\preceq (x', y')$ in $\mathbf{P}\times \mathbf{Q}$ if $x
\preceq x'$ in $\mathbf{P}$ and $y \preceq y'$  in $\mathbf{Q}$.


Let $\mathbf{L}$ be an FDL with the greatest element $\hat{1}$ and
the least element $\hat{0}$. If $\mathbf{L}$ can be expressed as the
direct product of a series of FDLs $\mathbf{L}_{j} (j\in J)$, i.e.
$\mathbf{L}= \prod_{j\in J}\mathbf{L}_{j},$ then we say that
$\mathbf{L}$ has a {\em (direct product) decomposition} $\prod_{j\in
J}\mathbf{L}_{j}$. A lattice with exactly one element is viewed as a
trivial lattice. An FDL is {\em irreducible} if it cannot be
expressed as direct product of  at least two non-trivial FDLs. A
decomposition $\mathbf{L}=\prod_{j\in J}\mathbf{L}_{j}$   is called
{\em irreducible}  if each $\mathbf{L}_{j} (j\in J)$ is non-trivial
and irreducible.

For a   decomposition $\mathbf{L}=\prod_{i=1}^n\mathbf{L}_{i}$, let
$\hat{1}_{i}$ and $\hat{0}_{i}$ denote the greatest element and the
least element of $\mathbf{L}_{i}$, respectively. Then $\hat{0}=
(\hat{0}_{1},\hat{0}_{2}, \cdots, \hat{0}_{n})$ and $\hat{1}=
(\hat{1}_{1},\hat{1}_{2}, \cdots, \hat{1}_{n})$.
\begin{figure}[h]
\refstepcounter{figure} \label{fig01}
\begin{center}
\includegraphics[width=6cm,keepaspectratio]{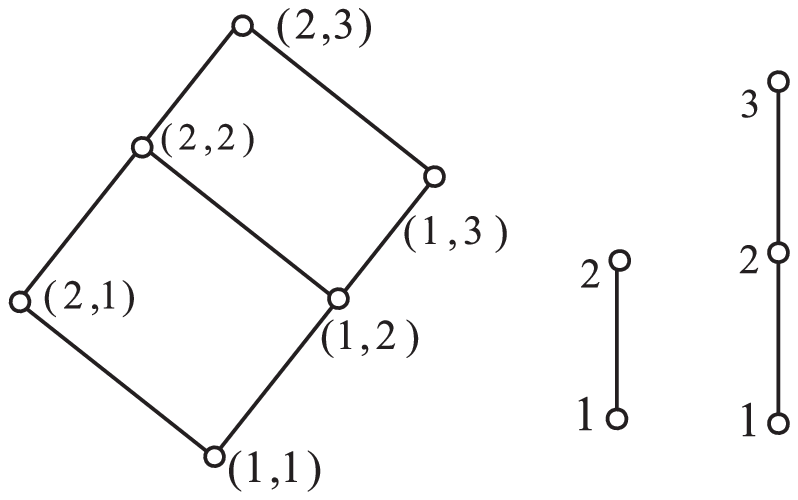}\\
Figure 1. A pair of central elements $(2,1)$ and $(1,3)$ of direct
product $\mathbf 2\times \mathbf 3$.
\end{center}
\end{figure}
If  each $\mathbf{L}_i$ is non-trivial and  $n \geq 2$,
$e_i=(\hat{0}_{1}, \cdots, \hat{0}_{i-1}, \hat{1}_{i},
\hat{0}_{i+1}, \cdots, \hat{0}_{n})$  is called a {\em central
element} of $\mathbf{L}$. For $x, y\in \mathbf{L}$, $x$ is called a
complement of $y$ if $x\vee y= \hat{1}$ and $x\wedge y= \hat{0}$.
The complement of $x$, when it exists,  is unique. For example, two
central elements $(2,1)$ and $(1,3)$   of $\mathbf{L}=\mathbf
2\times \mathbf 3$ are complementary each other (see Fig. 1). For a
positive integer $n$,   $\{1,2,...,n\}$ with its usual
 order forms an $n$-element chain, denoted by $\mathbf{n}$.

\begin{Lemma}\cite{gb}\label{unique complement}
Any central element of an FDL  has a unique complement.
\end{Lemma}

\begin{Lemma}\cite{gb}\label{permutation}
Any FDL has a unique irreducible decomposition, i.e. if both
$\prod_{i=1}^n\mathbf{L}_i$ and $\prod_{j=1}^m\mathbf{L}_j'$ are
irreducible decompositions of $\mathbf{L}$, then $m=n$ and there
exists a permutation $\pi$ of $[n]$ such that
$\mathbf{L}_{i}=\mathbf{L}'_{\pi(i)} (i=1, 2,\cdots, n)$.
\end{Lemma}

For  an FDL $\mathbf{L}$, its rank function \cite{ec} satisfies
$$\rho(x)+ \rho(y)= \rho(x\wedge y)+ \rho(x\vee y),$$ for any $x, y\in
\mathbf{L}$. For a pair of complementary elements $x$ and  $y$ of
$\mathbf L$, we have
$$\rho(x)+\rho(y)=
 \rho(\hat{0})+ \rho(\hat{1})= \rho(\hat{1})= \rho(\mathbf{L}).$$

\begin{Lemma}\label{sublattice}

Let $\mathbf{L}$ be an FDL of rank $k$ and $y$ the complement of
$x\in \mathbf{L}$. If $\rho(x)= r\geq 1$ and $\rho(y)= k-r\geq 1$,
then $\mathbf{L}$ has a sublattice
$(\mathbf{r+1})\times(\mathbf{k-r+1})$ containing $x$ and $y$.
\end{Lemma}

\noindent{\em Proof.}  $\mathbf{L}$ has at least two saturated
chains between $\hat{0}$ with $x$ and $y$, respectively:
$$ P_{1}: \hat{0}=x_0\prec x_{1}\prec x_{2}\prec \cdots \prec
x_{r}= x,$$ and $$ P_{2}: \hat{0}=y_0\prec y_{1}\prec y_{2}\prec
\cdots \prec y_{k-r}= y.$$ Then $x_{i}\wedge y_{j}= \hat{0}$, for
any $0\leq i\leq r, 0\leq j\leq k-r$, since $x_{i}\wedge
y_{j}\preceq x\wedge y=\hat{0}$. Hence $P_1$ and $P_2$ have no
common elements except for $\hat{0}$.

Let $\mathbf{L}'= \{a_{ij}: a_{ij}= x_{i}\vee y_{j}, 0\leq i\leq r,
0\leq j\leq k-r\}$. Then $\mathbf{L}'$ satisfies the following three
properties:
\begin{enumerate}
\item $a_{ij}=a_{i'j'}$ if and only if $i=i'$ and $j=j'$.  If $a_{i'j'}\preceq
a_{ij}$, i.e. $x_{i'}\vee y_{j'}\preceq x_{i}\vee y_{j}$, then
$(x_{i'}\vee y_{j'})\wedge y_{j'}\preceq (x_{i}\vee y_{j})\wedge
y_{j'}$. By distributive laws and $x_{i}\wedge y_{j}= \hat{0}$, we
have that  $y_{j'}\preceq y_{j}\wedge y_{j'}$ and $j'\leq j$.
Similarly we have $i'\leq i$. So the property holds.

\item $\mathbf{L}'$ forms a sublattice of $\mathbf{L}$
and $\mathbf{L}'= \langle x_{1}, \cdots, x_{r}, y_{1}, \cdots,
y_{k-r}; \vee, \wedge\rangle$. It suffices to show that
$\mathbf{L}'$ is closed under meet and join operations $\wedge$ and
$\vee$ of $\mathbf{L}$.



(i) $a_{ij}\vee a_{i'j'}= (x_{i}\vee y_{j})\vee (x_{i'}\vee y_{j'})=
(x_{i}\vee x_{i'})\vee (y_{j}\vee y_{j'})= x_{i''}\vee y_{j''}=
a_{i''j''}\in \mathbf{L}'$, where $i''= \max\{i, i'\}, j''= \max\{j,
j'\}$;

(ii) $$\begin{array}{rll} a_{ij}\wedge a_{i'j'}&=& a_{ij} \wedge
(x_{i'}\vee y_{j'})= (a_{ij}\wedge x_{i'})\vee (a_{ij}\wedge
y_{j'})\\
&=&((x_{i}\vee y_{j})\wedge x_{i'})\vee ((x_{i}\vee y_{j})\wedge
y_{j'})\\
&=&((x_{i}\wedge x_{i'})\vee (y_{j}\wedge x_{i'}))\vee ((x_{i}\wedge
y_{j'})\vee ( y_{j}\wedge y_{j'}))\\
&=&((x_{i}\wedge x_{i'})\vee \hat{0})\vee (\hat{0} \vee
(y_{j}\wedge y_{j'}))\\
 &= &x_{i'''}\vee y_{j'''}=
a_{i'''j'''}\in \mathbf{L}',\end{array}$$

where $i'''= \min\{i, i'\}, j'''= \min\{j, j'\}$.


\item 
$\mathbf{L}'$ is isomorphic to
$(\mathbf{r+1})\times(\mathbf{k-r+1})$. Let $\phi:
\mathbf{L}'\rightarrow (\mathbf{r+1})\times(\mathbf{k-r+1})$ be a
bijection  as $\phi(a_{ij})=(i,j)$ for any $a_{ij}\in \mathbf L'$.
Then by Property 2(i) and (ii), we have that
$$\phi(a_{ij}\vee a_{i'j'})=\phi( a_{i''j''})=(i'',j'')= (i,j)\vee (i',j')=\phi(a_{ij})\vee \phi(a_{i'j'}),$$ and
$$\phi(a_{ij}\wedge a_{i'j'})=\phi( a_{i'''j'''})=(i''',j''')=(i,j)\wedge
(i',j')= \phi(a_{ij})\wedge \phi( a_{i'j'}).$$ Hence
$\mathbf{L}'\cong (\mathbf{r+1})\times(\mathbf{k-r+1})$. \hfill
$\square$
\end{enumerate}

\section{Some fundamental results on MDL}

Let $G$ be a plane bipartite graph with a specific proper
black-white coloring to vertices. An edge $e$ of a cycle (or an
inner face) $C$ is {\em proper} if $e$  goes from the white
end-vertex  to the black endvertex  along the clockwise direction of
$C$. Let $\cal F(G)$ denote the set of all inner faces of $G$.
Recall that ${\cal M}(G)$ denotes the set of all 1-factors of $G$.
\vskip 2mm

\noindent{\bf Definition 1}. A binary relation $\preceq$ on ${\cal
M}(G)$ is defined as: $M_1\preceq M_2$,  $M_1,M_2\in {\cal M}(G)$,
if and only if $\vec Z(G)$ has a directed path from $M_2$ to $M_1$.
\vskip 2mm

 It is known that  $\mathbf M(G)=(\cal
M(G),\preceq)$ is a poset and a lattice structure on $\cal M(G)$ is
revealed in the following two theorems.

\begin{thm}\cite{lz}\label{webg}
Let $G$ be a plane (weakly) elementary bipartite graph. Then
$\mathbf{M}(G)$ is a finite distributive lattice, and its Hasse
diagram is isomorphic to $\vec{Z}(G)$.
\end{thm}

\begin{thm}\cite{zh1}\label{nwebg}
Let $G$ be a plane bipartite graph with 1-factor.   Then $\mathbf
M(G)$ is direct sum of distributive lattices and the Hasse diagram
is isomorphic to $\vec Z(G)$.
\end{thm}

\noindent{\bf Definition 2}. An FDL $\mathbf{L}$ is called an {\em
matchable distributive lattice} (MDL) if there exist a plane
bipartite graph $G$ such that  $\mathbf{L} \cong \mathbf{M}(G)$.
\vskip 2mm

Let $M^{\hat{1}}$ and $M^{\hat{0}}$ denote 1-factors of $G$ such
that $G$ has neither  improper $M^{\hat{1}}$- nor proper
$M^{\hat{0}}$-alternating cycles, called {\em source} and {\em root}
1-factors of $G$ respectively. If $\mathbf M(G)$ is an FDL, then
$M^{\hat{1}}$ and $M^{\hat{0}}$ are the greatest element and the
least element, respectively.

\begin{Lemma} \label{altbound}
Let $G$ be a plane elementary bipartite graph with more than two
vertices. Then the boundary of $G$ is proper $M^{\hat{1}}$- and
improper $M^{\hat{0}}$-alternating cycle.
\end{Lemma}

\begin{proof}  It is known that $G$ is 2-connected and the boundary is a cycle. For every
proper edge $e=uv$ on the boundary of $G$,  it suffices to show that
$e\in M^{\hat{1}}$. Otherwise, an edge $e'$ different from $e$ and
incident to $u$ belongs to $M^{\hat{1}}$. Since $G$ is elementary,
it has a 1-factor $M$ such that $e\in M$. Then $M\oplus M^{\hat{1}}$
has a cycle containing $e$ and $e'$, which is  both improper
$M^{\hat{1}}$- and proper $M$-alternating cycle, a contradiction.
Hence the boundary of $G$ is proper $M^{\hat{1}}$-alternating cycle.
Similarly, we can show that the boundary of $G$ is improper
$M^{\hat{0}}$-alternating cycle.
\end{proof}

 Let $G$ be a plane elementary bipartite graph with  $M'\preceq M$ in $\mathbf M(G)$.  For any $f\in \cal F(G)$, let $\Delta_{\cal C}(f)$ denote the number
 of proper $M$-alternating cycles in $\cal C:=\cal C(M,M')=M\oplus M'$ with $f$ in their
 interiors minus the number
 of improper $M$-alternating cycles in $\cal C$ with $f$ in their
 interiors. Then  $M'\preceq M$ implies that $\Delta_{\cal C} (f)\geq 0$ by Lemma 3.1 in \cite{zh1}. For any directed  path $\vec P= M_0(= M)M_1 ...M_t(= M')$ from $M$ to $M'$ of $\vec Z(G)$, let $s_i := M_{i-1}\oplus M_i$, $i =1, ..., t-1$. Let $\delta_P(f)$ denote
the times of $f$ appearing in the face sequence corresponding to
$s_1, ..., s_t$. Lemma 3.5 in Ref. \cite{zh1} implies the following
result.

\begin{Lemma} \label{delta}Let G be a plane elementary bipartite graph with  $M'\preceq M$ in $\mathbf M(G)$. If $\vec P$ is a directed path from $M$ to $M'$ of $\vec Z(G)$ and $\cal C=M\oplus M'$, then
$\delta_P (f) = \Delta_{\cal C}(f)$ for each $f\in \cal F(G)$.

\end{Lemma}

 From Lemmas \ref{sublattice}  and \ref{altbound}
we can derive the following critical result.

\begin{Lemma}\label{no complement}
For a plane elementary bipartite graph $G$ with more than two
vertices, each element of $\mathbf{M}(G)$ has no complement except
the greatest element $M^{\hat{1}}$ and the least element
$M^{\hat{0}}$.
\end{Lemma}

\setlength{\unitlength}{1cm}

\begin{figure}[h]
\refstepcounter{figure} \label{fig02}
\begin{center}
\includegraphics[scale=0.7]{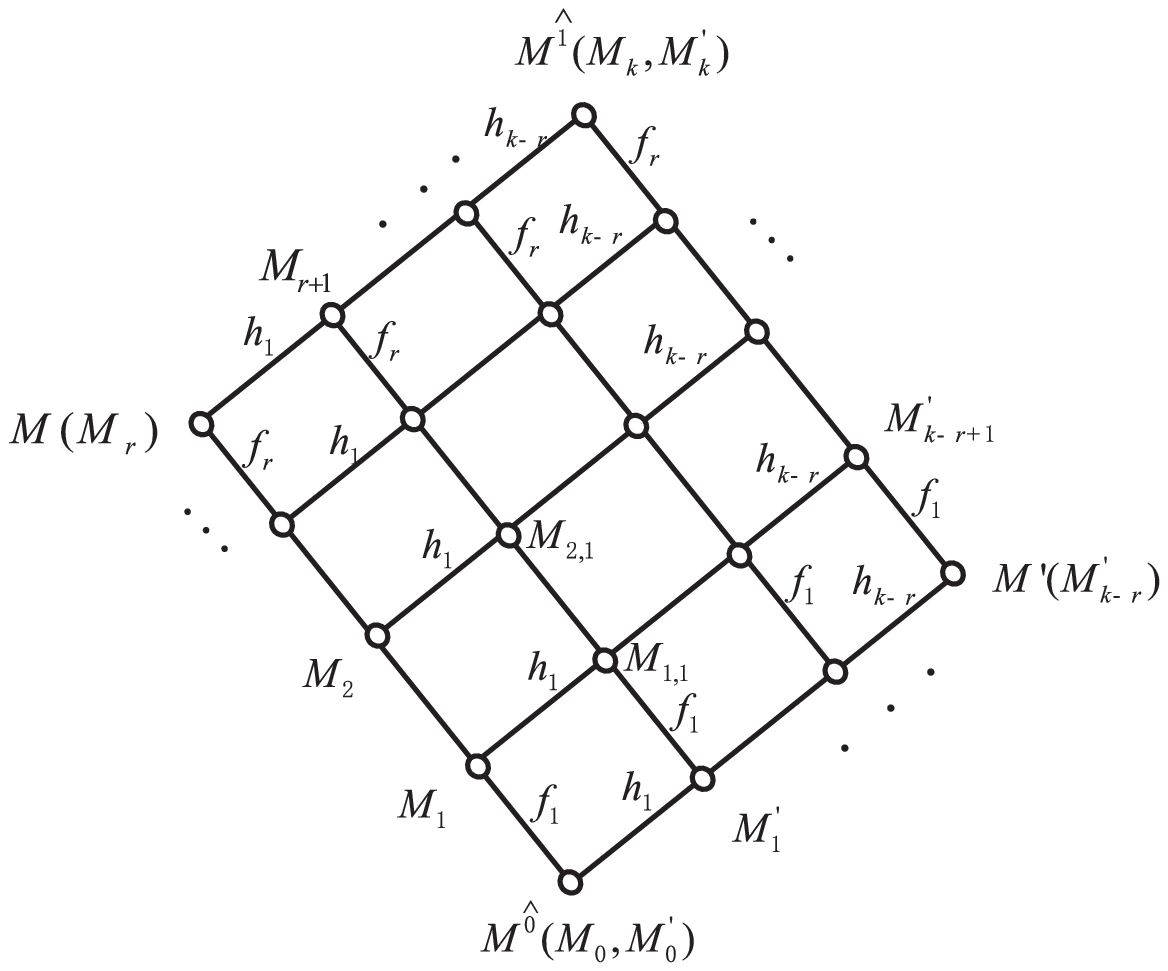}\\
Figure 2. Sublattice $(\mathbf{r+1})\times(\mathbf{k-r+1})$.
\end{center}
\end{figure}

\begin{proof}  Suppose to the contrary that $\mathbf{M}(G)$
has a pair of mutually complementary  elements $M$ and $M'$ except
$M^{\hat{1}}$ and $M^{\hat{0}}$. Let $\rho(\mathbf{M}(G))= k$ and
$\rho(M)= r$. Then $\rho(M')=k-r$, and $k-1\geq r\geq 1$. By Lemma
\ref{sublattice}, $\mathbf{M}(G)$ has a sublattice
$(\mathbf{r+1})\times(\mathbf{k-r+1})$ as shown in Fig.  2
containing the following two maximal chains:
$$M_{0}(=M^{\hat{0}})\prec M_{1}\prec \cdots \prec M_{r}(=M)\prec \cdots\prec
M_{k-1}\prec M_{k}(=M^{\hat{1}}),\mbox{ and}$$ $$ M'_{0}(=
M^{\hat{0}})\prec M'_{1}\prec \cdots \prec M'_{k-r}(=M')\prec
\cdots\prec M'_{k-1}\prec M'_{k}(=M^{\hat{1}}).$$

\noindent Put $M_{ij}:= M_i\vee M'_{j},\ i=0,1,\ldots,r, \,
j=0,1,\ldots,k-r,$ $s_{i}:=M_{i}\oplus M_{i-1} (1\leq i\leq r)$ and
$s'_{j}:=M'_{j}\oplus M'_{j-1} (1\leq j\leq k-r)$. Since each
maximal chain of $(\mathbf{r+1})\times(\mathbf{k-r+1})$ is a
saturated chain of $\mathbf M(G)$, the $s_{i}$ and $s'_{j}$ are the
boundaries of inner faces of $G$.\vskip 2mm

\noindent{\bf Claim 1}.  $M_{i,j}=M_{i-1,j}\oplus
s_{i}=M_{i,j-1}\oplus s'_{j}$, and $s_i$ and $s_j'$ are disjoint,
for $ i=1,2,\ldots,r$, and $j=1,2,\ldots,k-r$.
\begin{proof}

We prove that $M_{i,j}=M_{i-1,j}\oplus s_{i}=M_{i,j-1}\oplus s'_{j}$
such that $s_i$ and $s_j'$ are proper $M_{i,j}$-alternating by
induction on $(i,j)\geq (1,1)$. For $i=j=1$, $s_1$ and $s_1'$ are
improper $M^{\hat 0}$-alternating facial cycles, and are thus
disjoint.  Hence we have that $M_{1,1}=M_1\oplus s_1'=M_1'\oplus
s_1$ by Lemma \ref{sublattice}, and the required holds. Let $i\geq
2$ or $j\geq 2$. For the induction step, suppose that the assertion
holds for smaller $i$ or $j$.  By induction hypothesis,
$M_{i,j-1}=M_{i-1,j-1}\oplus s_i$ and $M_{i-1,j}=M_{i-1,j-1}\oplus
s'_j$, and  $s_i$ and $s'_j$ are distinct and improper
$M_{i-1,j-1}$-alternating facial cycles, and are disjoint. Hence
$s_i$ and $s'_j$ are improper $M_{i,j-1}$-alternating  and improper
$M_{i-1,j}$-alternating, respectively, and  $M_{i,j-1}\oplus s_j'$
and $M_{i-1,j}\oplus s_i$ cover $M_{i,j-1}$ and  $M_{i-1,j}$,
respectively. Obviously, $M_{i,j-1}\oplus s'_{j}=M_{i-1,j-1}\oplus
s_i\oplus s'_j=M_{i-1,j-1}\oplus s'_j\oplus s_i=M_{i-1,j}\oplus
s_i$. Hence $M_{i,j}=M_{i-1,j}\oplus s_i=M_{i,j-1}\oplus s'_{j}$
since $M_{i,j}=M_{i-1,j}\vee M_{i,j-1}$ by Lemma \ref{sublattice}.
The assertion holds for any $(i,j)$.\end{proof}

\noindent{\bf Claim 2.} Let $f_{i}$ and $h_{j}$ denote the inner
faces of $G$ bounded by $s_i$ and $s_j'$, respectively. Then $\cal
F(G)=\{f_{1}, f_{2}, \ldots, f_{r}, h_{1}, h_{2}, \ldots,
h_{k-r}\}$.

\begin{proof}  Let $\mathcal{F}_{1}:=\{f_{1},
f_{2}, \cdots, f_{r}\}$ and $ \mathcal{F}_{2}:=\{h_{1}, h_{2},
\cdots, h_{k-r}\}$.  So we want to prove that
$\mathcal{F}(G)=\mathcal{F}_{1} \cup \mathcal{F}_{2}$.

Let ${\cal C}:=M^{\hat{1}}\oplus M^{\hat{0}}$. Then each cycle in
${\cal C}$ is proper $M^{\hat{1}}$ and improper
$M^{\hat{0}}$-alternating cycle, one being the boundary of $G$ by
Lemma \ref{altbound}.  Hence $\Delta_{\cal C}(f)\geq 1$ for any
$f\in \cal F$.

Let $P:=M_{r,k-r}(=M^{\hat{1}})M_{r,k-r-1}\cdots
M_{r,0}M_{r-1,0}\cdots M_{0,0}(=M^{\hat{0}})$ be a directed path of
$\vec Z(G)$, corresponding to  a maximal chain of $\mathbf M(G)$.
For any $f\in \cal F$, by Lemma \ref{delta} we have that
$\delta_P(f)=\Delta_{\cal C}(f)\geq 1$. Hence $\cal F(G)=\{f_{1},
f_{2}, \ldots, f_{r}, h_{1}, h_{2}, \ldots, h_{k-r}\}$.
\end{proof}

 Since $G$ is  2-connected,  inner dual graph $G^\#$ of $G$ is
 connected. Let $f^*$ be a vertex of $G^\#$ corresponding to $f\in
 \cal F$.  Then there must exist a vertex $f_i^*$ in
 $\{f_1^*,\ldots, f_r^*\}$ being adjacent to a vertex $h_j^*$ in
 $V(G^*)\setminus \{f_1^*,\ldots, f_r^*\}=\{h_1^*,\ldots,
 h_{k-r}^*\}$.
That means that $f_i$ and $h_j$ are adjacent, contradicting Claim 1.
  \end{proof}


From the above arguments, we have the following main results of
this paper.

\begin{thm}\label{ is irreducible}
For a plane elementary bipartite graph $G$, $\mathbf{M}(G)$ is
irreducible.
\end{thm}

\begin{proof}  If $G=K_2$, it is trivial. Otherwise, $\mathbf{M}(G)$ is
a non-trivial FDL. By Lemma \ref{no complement}, every element of
$\mathbf{M}(G)$ has no complement except for $M^{\hat{1}}$ and
$M^{\hat{0}}$. By Lemma \ref{unique complement}, $\mathbf{M}(G)$ has
no central elements. Hence, $\mathbf{M}(G)$ is irreducible.
\end{proof}

{\em Elementary components} of a plane bipartite graph $G$ with
1-factor mean components other than $K_2$ of the subgraph obtained
from $G$ by the removal of all forbidden  edges (those edges  not
contained in any 1-factors).

\begin{cor}\label{ird}
Let $G$ be a weakly elementary plane bipartite graph with
 elementary components
 $G_1, G_2, \ldots, G_k$. Then $\mathbf
M(G)=\mathbf M(G_1)\times \mathbf M(G_2)\times \cdots \times \mathbf
M(G_k)$ is an irreducible decomposition.
\end{cor}

\begin{thm}({\bf Decomposition Theorem})\label{iff}
Let $\mathbf{L}$ be an FDL with a decomposition
$\mathbf{L}=\prod\limits_{i=1}^n\mathbf{L}_i$. Then $\mathbf{L}$ is
an MDL if and only if each $\mathbf{L}_i (1\leq i\leq n)$ is an MDL.
\end{thm}

\begin{proof}    If each factor $\mathbf{L}_{i}
$  is an MDL, $1\leq i\leq n$, then there exists a weakly elementary
plane bipartite graph $G_i$ such that $\mathbf{M}(G_i)\cong
\mathbf{L}_{i}$. We construct a weakly elementary plane bipartite
graph $G$ by connecting $G_i$ to $G_{i+1}$ with a new edge in their
exteriors for each $1\leq i\leq n-1$. Then  such new edges are
forbidden edges  of $G$.
It follows that $\mathbf M(G)\cong\mathbf M(G_1)\times \mathbf
M(G_2)\times \cdots \times \mathbf M(G_n)\cong\mathbf L_1\times
\mathbf L_2\times \cdots \times \mathbf L_n=\mathbf L$. Hence
$\mathbf L$ is an MDL.

Conversely, suppose that $\mathbf L$ is an MDL. Then there exists a
plane weakly elementary bipartite graph $G$ such that
$\mathbf{M}(G)\cong \mathbf{L}$.  Let  $G_1, \cdots, G_m$ be the
non-trivial elementary components of $G$  ($m\geq 1$). By Corollary
\ref{ird} $\mathbf L\cong\prod_{j=1}^m\mathbf{M}(G_j)$ is an
irreducible decomposition.  If
$\mathbf{L}=\prod_{i=1}^n\mathbf{L}_i$ is irreducible, then by Lemma
\ref{permutation}, $m= n$ and there exists a permutation $\pi$ of
$[n]$ such that $\mathbf{L}_i=\mathbf{M}(G_{\pi(i)})$ $i=1, 2,
\cdots, n$. So each $\mathbf{L}_i$ ($1\leq i\leq n$) is an MDL. If
$\prod_{i=1}^n\mathbf{L}_i$  is not irreducible, then each factor
$\mathbf{L}_i$ is a direct product of some $\mathbf{M}(G_j)$'s. So
each factor $\mathbf L_i$ is still an MDL.
\end{proof}

\section{MDL $\mathbf J(\mathbf{m}\times \mathbf{n})$}

From now on  we will present two typical irreducible MDLs by the
{\em fundamental theorem for finite distributive lattice} (FTFDL).

Let $\mathbf{P}$ be a finite poset. An {\em order ideal} ({\em
semi-ideal} or {\em down-set}) $\mathbf{I}$ of  $\mathbf{P}$ is a
subset  of $\mathbf{P}$  if for every $x\in \mathbf{I}$,  $y\preceq
x$ implies $y\in \mathbf{I}$.
The set $J(\mathbf{P})$ of order
ideals of  $\mathbf{P}$, ordered by the set-inclusion, forms a poset
$\mathbf{J(P)}$. 
It is well known that $\mathbf{J(P)}$ is indeed a distributive
lattice.  The FTFDL states that the converse is true.

\begin{thm}[FTFDL](\cite{ec})\label{ftfdl}
Let $\mathbf{L}$ be an FDL. Then there is a unique (up to
isomorphism) finite poset $\mathbf{P}$ for which $\mathbf{L}\cong
\mathbf{J(P)}$.
\end{thm}

In fact the above $\mathbf{P}$ can be viewed as a subposet of
$\mathbf{L}$ consisting of all join-irreducible elements of
$\mathbf{L}$: an element $x$ of $\mathbf{L}$ is said  to be {\em
join-irreducible} if one cannot write $x=y \vee z$ where $y\prec x$
and $z\prec x$.

\begin{figure}[!htbp]
\refstepcounter{figure} \label{fig03}
\begin{center}
\includegraphics[width=14cm,keepaspectratio]{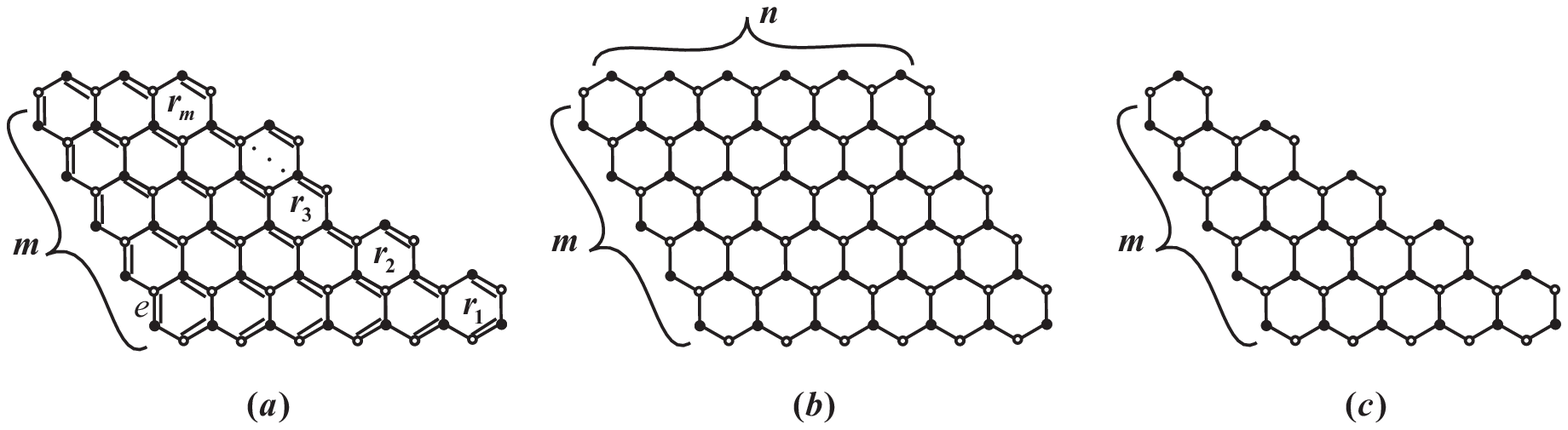}\\
Figure  \ref{fig03}. {(a)Truncated parallelogram $H=L(r_1, r_2,
\cdots, r_m)$ with root 1-factor $M^{\hat{0}}$,\\ (b) parallelogram
$L(m;n)$, and (c)prolate triangle $T_m$.}
\end{center}
\end{figure}
 In this section we show that $\mathbf J(W)$  are MDLs for any order ideal $W$ of $\mathbf{m}\times \mathbf{n}$.
Let us introduce a type of hexagonal systems called  truncated
parallelogram  \cite{bcgksb,ecgakb}:   A \emph{truncated
parallelogram}, simply denoted  by $H:=L(r_1, r_2, \cdots, r_m)$,
consists of $m$ condensed linear chains (rows) of the length $r_1,
\cdots, r_m, \, r_1 \ge r_2 \ge \cdots \ge r_m
>0$ and the first hexagons (conventionally drawn to the left) from
all chains also form a linear chain, the first column; In
particular, $L(m;n)=L(n,n,...,n)$  is  a \emph{parallelogram}, and
$T_m:=L(m,m-1,\cdots,1)$ is a \emph{prolate triangle}. For example,
see Fig.  \ref{fig03}. For convenience, all hexagonal systems
considered in this section are drawn such that an edge-direction is
vertical and the valleys are colored white.


Let $L$  and $B$ be the left and bottom perimeters of $H$,
respectively, which have a black vertex in common. The root 1-factor
 $M^{\hat{0}}$ of $H$ has all vertical edges in $L$, and a series
of parallel edges of $B$ from left-low to right-up, and a series of
parallel edges of $H-L-B$ from left-up to right-lower. We can see
that the boundary of $H$ is an improper $M^{\hat{0}}$-alternating
cycle. Hence $H$ is elementary \cite{zz2}.

 Since $H$ has a forcing edge $e$ (an edge contained in a unique
 1-factor), each $M^{\hat{0}}$-alternating cycle must pass through
 $e$; see \cite{zl} for details.  For each 1-factor $M$ of $H$ other than $M^{\hat{0}}$,  $C_M:=M \oplus M^{\hat{0}}$ is an
 $M^{\hat{0}}$-alternating cycle of $H$. Thus we have a bijection
 \cite{enumeration} between the 1-factors other than $M^{\hat{0}}$ of $H$
 and the  $M^{\hat{0}}$-alternating cycles of $H$.
 Hence the subhexagonal system of $H$ formed by $C_M$ together with
 its interior is also a truncated parallelogram.  Conversely, the perimeter of any
 sub-truncated parallelogram of $H$ with edge $e$ is an $M^{\hat{0}}$-alternating
 cycle. Hence each  1-factor $M$ of $H$ corresponds exactly to a
 sub-truncated parallelogram  of $H$ with edge $e$, denoted by $H_M$. However, $H_{M^{\hat{0}}}$
 corresponds to the empty graph (without vertex), the degenerated sub-truncated parallelogram of $H$.

 Let $P_M:=(L  \cup B)\oplus C_M$. Then $P_M$ is an $M$-alternating path with both end-edges in $M$ (see
 Fig.   \ref{fig04}(a)). Note that $C_{M^{\hat{0}}}=\emptyset$ and
 $P_{M^{\hat{0}}}=L \cup B$. From $M=
 M^{\hat{0}}\oplus C_M$, we have the following structure of $M$.

\begin{pro}\label{sfm}
For each  $M \in \mathcal{M}(H)$, the edges in $M \setminus E(P_M)$
have the same edge-direction from left-up to right-low. \hfill
$\Box$
\end{pro}

\begin{figure}[!htbp]
\refstepcounter{figure} \label{fig04}
\begin{center}
\includegraphics[width=14cm,keepaspectratio]{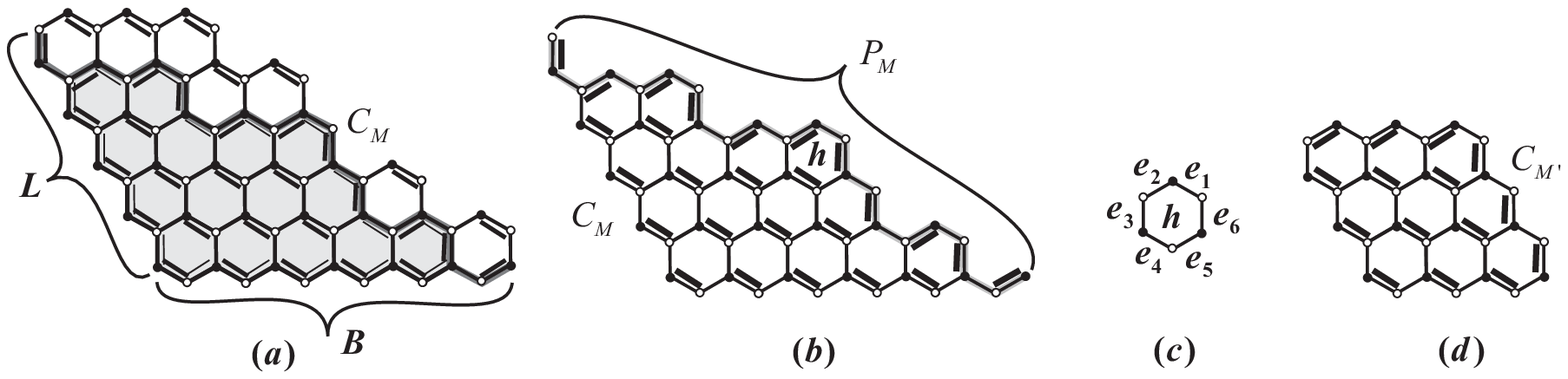}\\
Figure  \ref{fig04}. {(a) Truncated parallelogram $H$ with an
$M,M^{\hat{0}}$-alternating cycle $C_M$, (b)  $M$-alternating path
$P_M$,  (c) a hexagon $h$,  and (d) sub-parallelogram $H_{M'}$.}
\end{center}
\end{figure}

\begin{Lemma}\label{alth}
Let $M \in \mathcal{M}(H)$ and $h$  an $M$-alternating hexagon of
$H$. Then $h$ intersects at three consecutive edges  of $P_M$;
Moreover, $h$ is proper if and only if $h \subseteq H_M$.
\end{Lemma}

\begin{proof} If $h$ is disjoint with $P_M$, then $h$ is not $M$-alternating by Proposition \ref{sfm}.
Otherwise, $1\leq |E(h\cap P_M)|\leq 3$.  Since  $h$ is
$M$-alternating, $h$ intersects at three consecutive edges of $P_M$.
So $h$ is proper $M$-alternating if and only if $e_2, e_4, e_6 \in
M$. This holds if and only if $e_2, e_6 \in M \cap E(P_M)$. Thus $h$
and $P_M$ have exactly three common edges $e_1,e_2, e_6$ and $h \in
H_M$. Similarly, $h$ is improper $M$-alternating if and only if $h$
and $P_M$ have exactly three common edges $e_3, e_4,e_5$ and $h
\notin H_M$ (see Fig.  \ref{fig04} (b) and(c)).
\end{proof}

\begin{Lemma}\label{mhm}
Let $M, M' \in \mathcal{M}(H)$. Then $M' \preceq M$ in $\mathbf{M}(H)$ if and only if $H_{M'}$
is a sub-truncated parallelogram of $H_{M}$, namely $H_{M'} \subseteq H_{M}$.
\end{Lemma}
\begin{proof} We first show that $M$ covers $M'$ if and only if $H_M$ can be
obtained from $H_{M'}$ by adding a hexagon.
If $M$ covers $M'$, then there is a proper $M$-alternating hexagon
$h$ such that $M'=M \oplus h$ by Theorem \ref{webg}, and $C_{M'}= M'
\oplus M^{\hat{0}} = (M' \oplus M) \oplus (M \oplus M^{\hat{0}}) = h
\oplus C_M$.  By Lemma \ref{alth}, we have that $h \in H_M$ has
exactly three edges of $P_M$.  If $h=C_M$, the result is trivial.
Otherwise, $C_{M'}$ is an improper $M^{\hat{0}}$-alternating cycle,
and the sub-truncated parallelogram $H_{M'}$ of $H$ bounded by
$C_{M'}$ can be obtained by removing $h$ from $H_M$.

Conversely,  assume that $H_{M'}$ can be obtained from $H_M$ by
removing a hexagon $h$ of $H$. Since both $H_M$ and $H_{M'}$ are
sub-truncated parallelograms of $H$, $h \in H_M$ must have exactly
three edges of $P_M$. By Lemma \ref{alth}, $h$ is proper
$M$-alternating. Then $M'=M^{\hat{0}} \oplus C_{M'}=M^{\hat{0}}
\oplus (C_M \oplus h)=M \oplus h$. Hence $M$ covers $M'$ in
$\mathbf{M}(H)$.

We now show the lemma. If $M' \preceq M$ in $\mathbf{M}(H)$, we can
show that that $H_{M'} \subseteq H_{M}$ by choosing a saturated
chain between $M$ and $M'$ and  applying repeatedly the above fact
proved. If $H_{M'} \subseteq H_{M}$, there are a series of
sub-truncated parallelograms of $H$: $H_1(=H_{M'}), H_2, ...,
H_t(=H_M)$, such that each $H_i$ is obtained from $H_{i+1}$ by
adding a hexagon. Each $H_i$ corresponds to a 1-factor $M_i$ of $H$,
$i=1,2,...,t$. By the above fact we have $M_{i+1}$ covers $M_i$,
$i=1,2,...,t-1$. Hence $M'\preceq M$.
 \end{proof}

Now, we define a poset  on $\mathcal{F}(H)$, the set of hexagons of
$H$. $h \in H$ is labeled with $h_{ij}$ if $h$ lies in the $i$-th
row and $j$-th column, $1\le j \le r_i, 1\leq i\leq m$. For two
hexagons $h_{ij}$ and $h_{kl}$, $h_{ij} \preceq h_{kl}$ if and only
if $i \leq k$ and $j \leq l$. Then $\mathbf{F}(H):=(\mathcal{F}(H),
\preceq)$ is a poset. If $H$ is a parallelogram, then
$\mathbf{F}(H)$ is $\mathbf{m}\times \mathbf{n}$. In general,
$\mathbf{F}(H)$ is an order ideal of $\mathbf{m}\times \mathbf{n}$.
For example, see Fig.  \ref{fig05}.


\begin{figure}[!htbp]
\refstepcounter{figure} \label{fig05}
\begin{center}
\includegraphics[width=14cm,keepaspectratio]{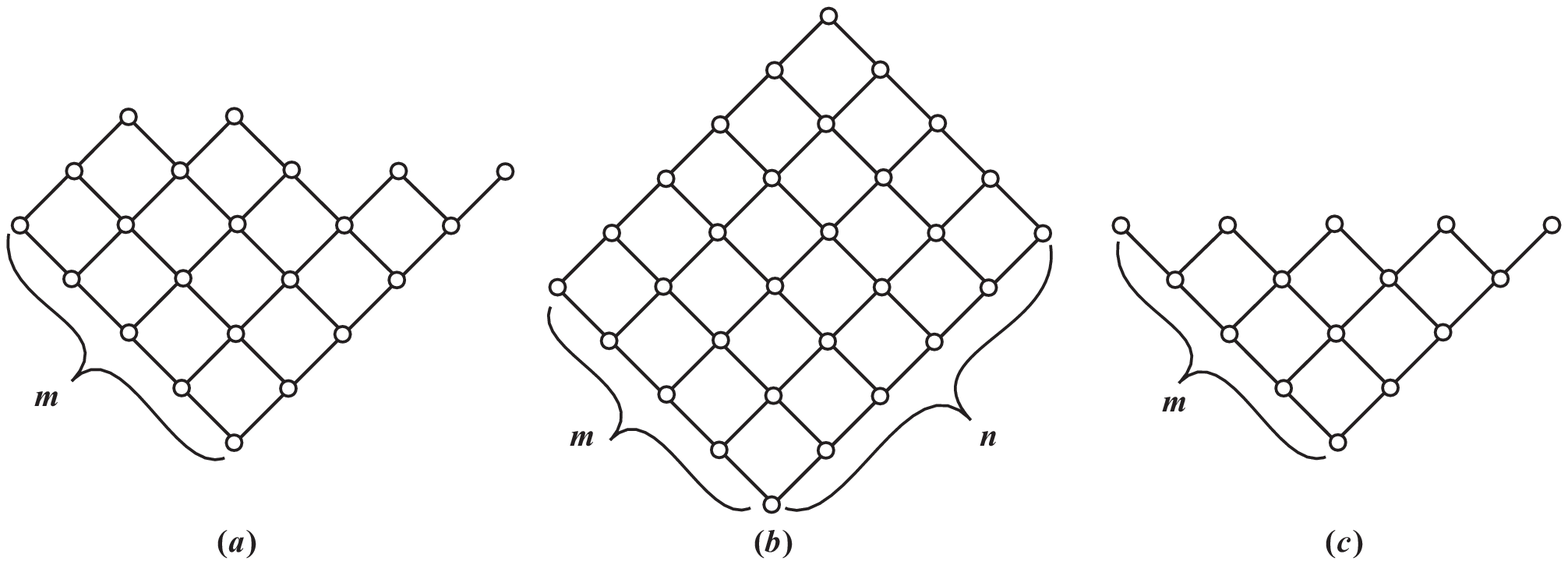}\\
Figure \ref{fig05}. (a) $\mathbf{F}(L(r_1, r_2, \cdots, r_m))$, (b)
$\mathbf{F}(L(m;n))$,  and (c) $\mathbf{F}(T_m)$.
\end{center}
\end{figure}

By Lemma \ref{mhm}, we can see that  $M\in \mathbf{M}(H)$ is
join-irreducible if and only if $H$ has a unique proper
$M$-alternating hexagon lying in $H_M$, which is  a
sub-parallelogram of $H$. Let $\mathcal{I}(\mathbf{M}(H))$ denote
the subposet of  $\mathbf{M}(G)$ consisting of all join-irreducible
elements.

\begin{Lemma}\label{ordpre}
 $\mathcal I(\mathbf{M}(H)) \cong \mathbf{F}(H)$.
\end{Lemma}
\begin{proof}  A bijection $\psi:
\mathcal{I}(\mathbf{M}(H))\rightarrow\mathbf{F}(H)$ is defined as
follows. For each  $M \in \mathcal{I}(\mathbf{M}(H))$, let $\psi(M)$
denote the unique proper $M$-alternating hexagon of $H_M$, i.e. the
right-up-most hexagon of $H_M$. Moreover, both $\psi$ is an
isomorphism: for any $M, M' \in \mathcal{I}(\mathbf{M}(H))$, by
Lemma \ref{mhm} we have that $M' \preceq M$ in $\mathbf{M}(H)
\Leftrightarrow H_{M'} \subseteq H_M$ $\Leftrightarrow$ $\psi(M')
\preceq \psi(M)$ in $\mathbf{F}(H)$.
\end{proof}

By Theorem \ref{ftfdl} and Lemma \ref{ordpre}, we have a  main
theorem as follows.

\begin{thm}\label{cmo}
$\mathbf{M}(H) \cong \mathbf{J}(\mathbf{F}(H))$.          \hfill $\Box$
\end{thm}

When  $H$  takes  all over  the truncated parallelograms for fixed
$m$ and $n$, $\mathbf{F}(H)$ goes all order ideals of $\mathbf{m}
\times \mathbf{n}$. From the above theorem we have an immediate
consequence as follows.

\begin{cor}
  Let $W$ be any order ideal of $\mathbf{m} \times \mathbf{n}$. Then $\mathbf{J(W)}$ is an irreducible MDL.         \hfill $\Box$
\end{cor}

We can obtain a series of MDLs by applying the above theorem to the
special truncated parallelograms, such as parallelogram, prolate
triangle, etc. Let  $\mathbf{R}_m:=\mathbf{F}(T_m)$.  Here we give
two special ones:

\begin{cor}\label{ctplmn} $\mathbf{J}(\mathbf{m} \times \mathbf{n})$
and $\mathbf{J}(\mathbf{R}_{m})$ are irreducible MDLs. Moreover
\begin{enumerate}
  \renewcommand{\labelenumi}{(\arabic{enumi})}
  \item $\mathbf{J}(\mathbf{m} \times \mathbf{n}) \cong
  \mathbf{M}(L(m;n))$, and
  \item $\mathbf{J}(\mathbf{R}_{m}) \cong \mathbf{M}(T_{m})$.        \hfill $\Box$
\end{enumerate}
\end{cor}

Note  that \cite{bcgksb,ecgakb} the number of 1-factors of
parallelogram $L(m;n)$ and prolate triangle $T_m$ are
$\binom{m+n}{m}$ and $\frac{1}{m+2}\binom{2m+2}{m+1}$ (Catalan
number), respectively.

\section{MDL $\mathbf J(\mathbf{T})$}

In this section we will show that $\mathbf J(\mathbf{T})$ is an
irreducible
 MDL with  outerplane bipartite graphs for a poset $\mathbf{T}$
 implied by any orientation of a tree. A connected plane graph $G$ is \emph{outerplane} if all vertices lie
on the boundary of the outer face of $G$. Let $\mathcal{G}$ be the
set of all 2-connected outerplane bipartite graphs. Catacondensed
hexagonal systems are typical members of $\mathcal{G}$ \cite{kz}.

An edge set $T$ of a connected graph $G$ is called a \emph{minimal}
edge-cut if $G-T$ is not connected but $G-T'$ remains connected for
any proper subset $T'$ of $T$. For a plane graph $G$,  let $e^*$ and
$f^*$ denote the  edge and vertex of dual graph $G^*$ corresponding
to edge $e$ and face $f$ of $G$, respectively; For $T \subseteq
E(G)$, put $T^*: = \{ e^* : e \in T \}$. Some edges in a plane graph
$G$ form a minimal edge-cut in $G$ if and only if the corresponding
dual edges form a cycle in $G^*$ \cite{brwigt}.  A minimal edge-cut
$T$ of a plane bipartite graph $G$ is called \emph{elementary edge
cut} (\emph{e-cut} for short) \cite{zz2} if all edges of \;$T$ are
incident with white vertices of one component of $G-T$, called the
\emph{white bank} of $T$, and the other component is the \emph{black
bank} of $T$.

\begin{Lemma}\cite{ZYY}\label{mt1}\
 Let $T$ be a minimal edge-cut of $G \in \mathcal{G}$. Then $T$ is
 an e-cut of \;$G$ if and only if for any 1-factor $M$ of \;$G$,
 $\mid M \cap T \mid =1.$
\end{Lemma}

We now give an orientation $\vec G^*$ of the dual $G^*$: an edge
$e^*$ is oriented as an arc from $f_1^*$ to $f_2^*$ if one goes
along $e^*$ from $f_1^*$ to $f_2^*$  the white end-vertex of $e$
lies right side.  For example, see Figs. 6 and 7. We can see that a
minimal edge-cut $T$ is an e-cut of $G$ if and only if $T^*$ forms a
directed cycle of $\vec G^*$.

For $G\in \mathcal{G}$, we now give a poset on $\mathcal{F}(G)$. Let
$\vec G^\#$ be the orientation of inner dual graph $G^\#$, obtained
from  directed dual graph $\vec G^*$ by deleting the vertex $f_0^*$
corresponding to the outer face of $G$. For $f_1,\,f_2\in
\mathcal{F}(G)$, we define ``$f_1\leq f_2$" if $\vec{G}^\#$ has a
directed path from $f_2^*$ to $f_1^*$. Since $\vec{G}^\#$ contains
no directed cycles, $\mathbf{F}(G):= (\mathcal{F}(G), \leq)$ is a
poset.

For a plane elementary bipartite graph $G$ with $M,M'\in
\mathbf{M}(G)$, if $M'\preceq M$, then there exists a saturated
chain  $M_0(=M)M_1 \cdots M_k(=M')$ in $\mathbf{M}(G)$ between $M$
and $M'$. Then  $M_{i-1} \text{ covers } M_{i}$, and $f_i:=M_{i-1}
\oplus M_i$ is a proper $M_{i-1}$-alternating face, $i=1,2, \cdots,
k$. Then we say that $M_i$ is obtained from $M_{i-1}$ by a
\emph{$Z$-transformation} on the (proper $M_{i-1}$-alternating) face
$f_i$, or simply by \emph{transforming} $f_i$. Further, we also say
that $M'$ is obtained from $M$ by a {\em Z-transformation sequence}
on inner faces $f_1, f_2, \cdots, f_k$, and $f_1, f_2, \cdots, f_k$
is a face sequence of $G$ by a  $Z$-transformations sequence
  from $M$ to $M'$. The $Z$-transformation of $G$ is {\em
simple} if every  inner face of $G$ is transformed at most once
during any Z-transformation sequence of $G$.

\begin{Lemma}\label{2opbg}
Let  $G \in \mathcal{G}$. If $f' \preceq f$ in
  $\mathbf{F}(G)$, then $f'$
  always appears after $f$ in any Z-transformation sequence of
  $G$ from $M^{\hat 1}$ to $M^{\hat 0}$. Hence $Z$-transformation of $G$ is
  simple.
\end{Lemma}

\begin{proof} 

  Without loss of generality,  suppose that $f$ covers $f'$ in $\mathbf F(G)$.
  That is, $(f,f')$ is an arc of directed inner dual $\vec G^\#$. Then $(f,f')$
  can be extended to a maximal directed path of $\vec
  G^\#$, which can be further extended to a directed cycle of $\vec G^*$, denoted by
  $\vec C:=f_0^*e_0^*f^*_1e_1^*...f_t^*e^*_tf_0^*$; see Fig. 6.
   Then $T=\{e_0,e_1,...,e_t\}$ is an e-cut of $G$, each edge $e_j$ is a common
   edge of $f_j$ and $f_{j+1}$ (subscript module $t+1$), and each $e_j$ is a proper edge
   of $f_{j+1}$ and improper edge of $f_j$, $0\leq j\leq t-1$.  For any $M \in
  \mathbf{M}(G)$, by Lemma \ref{mt1}, $\mid M \cap T \mid =1$.

  Let $P:=M_1M_2\cdots M_s$ be a directed path in $\vec Z(G)$ from $M^{\hat 1}$ to $M^{\hat 0}$.
  Then $\delta_P(f)=1$ for all $f\in \cal F(G)$ by Lemmas \ref{altbound} and \ref{delta}. Suppose that $M_{i+1}$ is obtained from $M_i$ by a Z-transformation on
  $f_j$. It is sufficient to show that $f_{j}, f_{j+1},..., f_t$ do dot
  appear in Z-transformations from $M^{\hat 1}$ to $M_i$. We proceed by induction on $j$. If $j=1$,
  then $e_0\in M_1, M_2,...,M_i$. Hence proper edge $e_{k-1}$ of
  each $f_{k}$, $k\geq 2$, does not belong to $M_1,M_2,...,M_i$.
 Hence, the required holds. By induction hypothesis we have that $f_{j+1},..., f_t$ do dot
  appear in Z-transformations from $M^{\hat 1}$ to $M_{i+1}$ through
    $M_i$. Suppose that $M_{i'+1}$ is obtained from $M_{i'}$ by a Z-transformation on
  $f_{j+1}$. Then $i+1\leq i'$, and  proper edge $e_j$ of $f_{j+1}$ belong
  to all $M_{i+1},..., M_{i'}$. That implies that proper edge $e_k$ of $f_{k+1}$ does not
  belong to $M_{i+1},..., M_{i'}$ for all $k>j$. Hence  $f_{j+1},..., f_t$ do dot
  appear in Z-transformations from $M^{\hat 1}$ to $M_{i'}$; that is, $f_{j+2},..., f_t$ do dot
  appear in Z-transformations from $M^{\hat 1}$ to $M_{i'+1}$ through $M_{i'}$, as  expected.
\end{proof}

\begin{figure}[!htbp]
\begin{center}
\refstepcounter{figure} \label{fig06}
\includegraphics[width=9cm,keepaspectratio]{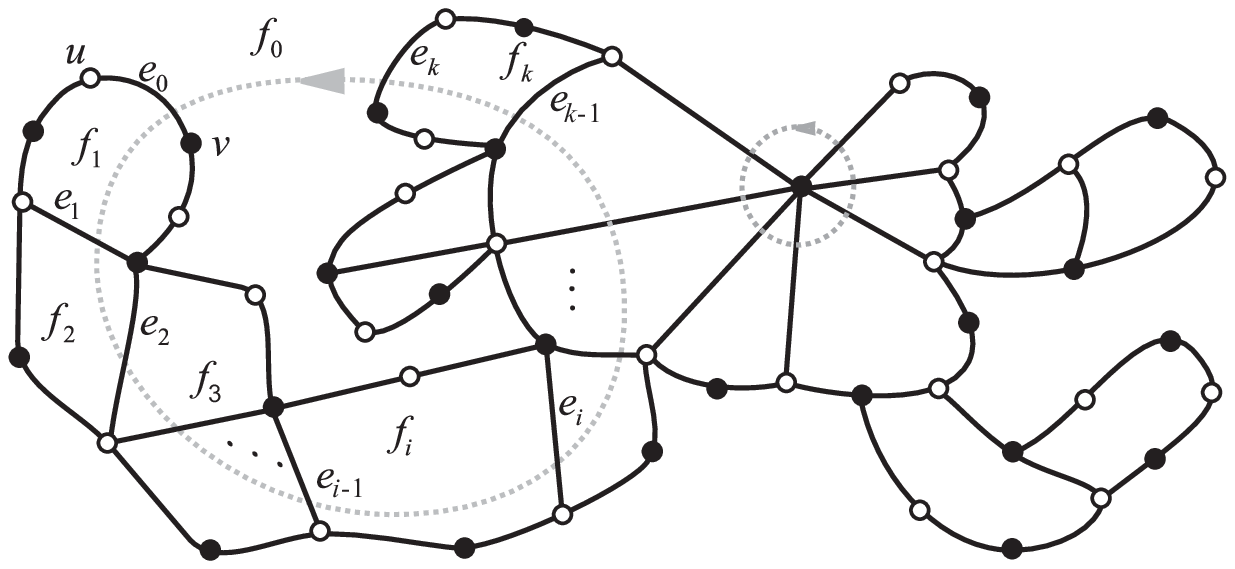}\\
Figure \ref{fig06}. { An outerplane bipartite graph with e-cuts (the
set of edges intersecting  a dashed line).  }
\end{center}
\end{figure}



For $G \in \mathcal{G}$, we now define a mapping from
$\mathcal{M}(G)$ to $J({\mathbf{F}(G)})$. For any $M \in
\mathcal{M}(G)$, let $\sigma(M)$ denote the set of faces in the face
sequence by a Z-transformation sequence  from $M$ to $M^{\hat{0}}$.
By Lemma \ref{delta}, we have
$$\sigma(M) =\{f \in \cal F(G) \mid f  \mbox{ is contained in the interior of some cycle in}\, M
\oplus M^{\hat{0}} \}.$$ In particular,
$\sigma(M^{\hat{0}})=\emptyset$, and $\sigma(M^{\hat{1}})=
\mathcal{F}(G)$ since $M^{\hat{1}} \oplus M^{\hat{0}}$ is just the
boundary  of $G$ by Lemma \ref{altbound}.

\begin{Lemma}$\sigma:
\mathcal{M}(G)\rightarrow {J}(\mathbf{F}(G))$ is an injective
mapping.
\end{Lemma}

\begin{proof}For $M\in \mathbf{M}(G)$, let $f\in \sigma (M)$. If $f'\prec f$ in $\mathbf
F(G)$, then  by Lemma \ref{2opbg}, $f'$
  always appears after $f$ in any Z-transformation sequence of
  $G$ from $M^{\hat{1}}$ to $M^{\hat{0}}$ passing through $M$. So $f'\in \sigma (M)$, and $\sigma(M)$ is an order ideal of
$\mathbf{F}(G)$. That is, $\sigma$ is  a mapping from
$\mathcal{M}(G)$ to $J({\mathbf{F}(G)})$. Further it is  clear that
$\sigma$ is injective.
\end{proof}

Further, we will show that $\sigma$ is an isomorphism between
$\mathbf{M}(G)$ and $\mathbf{J}(\mathbf{F}(G))$ in the following
theorem.

\begin{thm}\label{mcongj}
For each $G \in \mathcal{G}$,  $\mathbf{M}(G) \cong
\mathbf{J}(\mathbf{F}(G))$.
\end{thm}
\begin{proof} We first  show that, for
each $M \in \mathcal{M}(G)$, if $M_1, M_2, \cdots, M_k$ are the
1-factors covered by $M$, then the order ideals $\sigma(M_1),
\sigma(M_2), \cdots, \sigma(M_k)$ in $\mathbf{M}(G)$ are exactly the
ones covered by $\sigma(M)$. By the fact that the order ideals
 in a finite poset ${\mathbf P}$ covered in $J(P)$ by an order ideal $Y$ are exactly the
sets $Y \setminus \{ x \}$ for all maximal elements $x$ of $Y$, it
is sufficient to prove that the faces which can be properly
transformed in 1-factor $M$ are exactly the maximal elements of
$\sigma(M)$.

Let $f_i$ denote properly $M$-alternating facial cycle such that
$f_i=M\oplus M_i$, $i=1,2,...,k$.  For convenience, we also use
$f_i$ to denote the corresponding inner face.  Hence $\sigma(M_i) =
\sigma(M) \setminus \{ f_i \}$. By Lemma \ref{2opbg}, all faces
which greater than $f_i$ must be transformed during any transforming
sequence from $M^{\hat{1}}$ to $M$. So each $f_i$ is an maximal
element in $\sigma(M)$, and $\sigma(M_i) = \sigma(M) \setminus \{
f_i \}$ is covered by $\sigma(M)$.

If $f$ is a maximal element of $\sigma(M)$, then by Lemma
\ref{2opbg}, all elements which greater than $f$  in $\mathbf F(G)$
have been transformed from $M^{\hat{1}}$ to $M$.  Let $e$ be any
proper edge  of $f$. Let $f'$ be the face of $G$ that has a common
edge $e$ with $f$. If $f'$ is the outer face of $G$, by Lemma
\ref{altbound} $e$ remains unchanged in any Z-transformation from
$M^{\hat{1}}$ to $M$. Otherwise,  $f'\in \mathcal{F}(G)$ covers $f$,
and $f' \in \mathcal{F}(G) \setminus \sigma(M)$.  Then $e \in M$
since $f'$ has been transformed but $f$ not from $M^{\hat{1}}$ to
$M$ and $e$ is  an improper edge  of $f'$.  Hence, all proper edges
of $f$ belong to $M$; that is,  $f$ is proper $M$-alternating, as
expected.

Further, $\sigma$ is surjective since
$\mathcal{F}(G)=\sigma(M^{\hat{1}})$ is the maximum element of
$J(\mathbf{F}(G))$.  Therefore $\sigma$ is an isomorphism between
$\mathbf{M}(G)$ and $\mathbf{J}(\mathbf{F}(G))$.
\end{proof}

Given an undirected tree $T=(V, E)$, $\vec{T}$ is  any orientation
of $T$. Of course, directed tree $\vec{T}$ has no (directed) cycles.
Similar to $\mathbf{F}(G)$, we could consider $\vec{T}$ as the Hasse
diagram of a poset $\mathbf{T}$. As a consequence, we obtain another
irreducible MDL described in the following result.

\begin{thm}\label{jtree}
Let $\mathbf{T}$ be a poset derived from any orientation $\vec{T}$
of a tree $T$. Then $\mathbf J(\mathbf{T})$ is an irreducible MDL.
\end{thm}

\begin{proof} By Theorem \ref{mcongj}, it is sufficient
to show that there is a $G \in \mathcal{G}$ such that $\vec{G}^\#
\cong \vec{T}$. If $|V(T)|\leq 2$, it is obvious. So let $|V(T)|\geq
3$. Let $\Delta$ denote the maximum degree of $T$.
We now construct such a graph $G$ as follows. For any vertex $v$ of
$T$,  we gave an inner face $f_v$ bounded by a cycle of length
$2\Delta$.
  If a vertex  $u$ of $T$ is adjacent to $v$,  then we place the
  corresponding inner face $f_u$ outside $f_v$ by overlapping their edges $e'
\in f_u$ and $e'' \in f_v$ to a new edge $e \in G$, satisfying the
orientation rule of $\mathbf{F}(G)$:  $(u,v)$ is an arc from $u$ to
$v$ if and only if \;$e$ goes from the black end-vertex to the white
one  along the clockwise orientation of $f_u$.  Since $f_v$ has
$2\Delta$ edges and $v$ has at most $\Delta$ going-out (going-in)
arcs in the directed $\vec T$, for all other neighbors of $v$ we can
proceed similarly.  By repeating the above process,  one can
construct an outerplane bipartite graph $G\in \cal G$ such that
$\vec{G}^\# \cong \vec{T}$. For example, see Fig.  \ref{fig07}.
\end{proof}
\begin{figure}[!htbp]
\begin{center}
\refstepcounter{figure} \label{fig07}
\includegraphics[width=14cm,keepaspectratio]{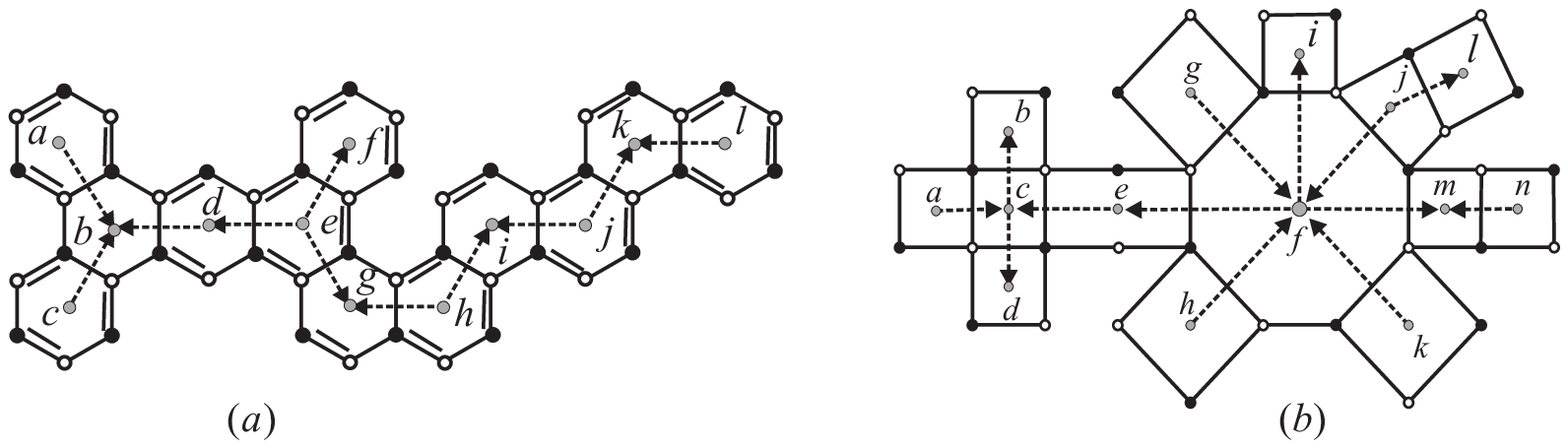}\\
Figure \ref{fig07}. {Two outerplane bipartite graphs with the
orientation for their inner duals. }
\end{center}
\end{figure}

In fact, in the above construction the face degree of $f_v$ may be
smaller than $2\Delta$.  The least value of face degree of $f_v$ may
reach $2\max \{\Delta_v^-,\Delta_v^+,2\}$, where
$\Delta_v^-,\Delta_v^+$ are in- or out-degree of $v$ in $\vec{T}$.
During the process, we may need to exchange the order of in- and
out-edges such that the edges surrounded $f_v$ are alternately in-
and out-edges as many as possible;  See Fig.  \ref{fig07} (b).




\end{document}